\documentclass[12pt]{article}

\usepackage{latexsym}
\usepackage{graphicx}
\usepackage{enumitem}
\usepackage{amsmath}
\usepackage{amssymb}
\usepackage{amsthm}
\usepackage{multirow}
\usepackage{color}
\usepackage{url}
\newtheorem{theorem}{Theorem}[section]
\newtheorem{definition}[theorem]{Definition}

\newtheorem{lemma}[theorem]{Lemma}

\newtheorem{problem}{Problem}
\newtheorem{conjecture}[problem]{Conjecture}

\title{A proof of purely singular splitting conjecture}

\author{Ka Hin Leung$^{\text{a}}$ and  Tao Zhang$^{\text{b,}}$\thanks{Corresponding author.}\\
	\footnotesize $^{\text{a}}$ Department of Mathematics, National University of Singapore, Kent Ridge, Singapore 119260, Republic of Singapore. \\
    \footnotesize $^{\text{b}}$ Institute of Mathematics and Interdisciplinary Sciences, Xidian University, Xi'an 710126, China.\\
	}
\begin{document}
	\date{}
\maketitle

\begingroup
\renewcommand\thefootnote{}% 去掉编号格式
\footnotetext{E-mail addresses: matlkh@nus.edu.sg (K. Leung), zhant220@163.com (T. Zhang).}
\addtocounter{footnote}{-1}% 不影响后续脚注编号
\endgroup

\begin{abstract}

A set $M$ of nonzero integers is said to split a finite abelian group $G$ if there exists a subset $S\subseteq G$ such that $M\cdot S = G\setminus\{0\}$. Such a splitting is called purely singular if every prime divisor of $|G|$ divides some element of $M$. In 1995,
Woldar \cite{W1995} conjectured that the finite abelian groups admitting a purely singular splitting by the set $\{1,2,\dots,k\}$ are precisely the cyclic groups of orders $1$, $k+1$, and $2k+1$. In this paper, we prove this conjecture.
	
	\medskip
	
	
	\smallskip
	
\end{abstract}

\section{Introduction}
Problems involving tilings of $\mathbb{R}^n$ by clusters of cubes have a long and rich history, dating back at least to the early work of Minkowski \cite{M1907}. In order to formalize this setting, let
\[Q=\{(x_1,\dots,x_n): 0\le x_i<1,\ x_i\in\mathbb{R}\}\]
denote the unit cube, which we regard as being placed at the origin.

Given a vector $e\in\mathbb{R}^n$, a unit cube placed at $e$ is defined as the translate
\[e+Q=\{e+x: x\in Q\}.\]
Building on this notion, a cluster of cubes is defined as a finite (or, more generally, discrete) union of pairwise disjoint translates of $Q$, namely
\[\mathcal{C}=\mathcal{B}+Q=\{e+Q: e\in\mathcal{E}\},\]
where $\mathcal{B}\subset\mathbb{R}^n$.

With this terminology in place, we can now describe packings and tilings. A set of pairwise disjoint translates of $\mathcal{C}$ is called a packing of $\mathbb{R}^n$ by $\mathcal{C}$. If, in addition, these translates cover the entire space $\mathbb{R}^n$, then the packing is said to be a tiling.

Furthermore, when the set of translation vectors forms an additive subgroup of $\mathbb{Z}^n$, the tiling is called a lattice tiling. This algebraic structure allows for a useful discrete reformulation: if $\mathcal{B}\subset\mathbb{Z}^n$, then a tiling of $\mathbb{R}^n$ by $\mathcal{C}$ is equivalent to a tiling of $\mathbb{Z}^n$ by $\mathcal{B}$.

Motivated by applications to error-correcting codes for flash memories, Wei and Schwartz \cite{WS22} introduced the notion of a limited-magnitude error ball. For integers $n\ge t\ge1$ and $k_{+}\ge k_{-}\ge0$, the $(n,t,k_{+},k_{-})$-error ball is defined as
\begin{align*}
\mathcal{B}(n,t,k_{+},k_{-})
:=\{{\bf a}=(a_1,a_2,\dots,a_n)\in\mathbb{Z}^n:\ a_i\in[-k_{-},k_{+}],\ \mathrm{wt}({\bf a})\le t\},
\end{align*}
where $\mathrm{wt}({\bf a})$ denotes the Hamming weight of ${\bf a}$.

As a first observation, when $n=t$, the set $\mathcal{B}(n,t,k_{+},k_{-})$ coincides with the cube $[-k_{-},k_{+}]^n$. Therefore, in what follows, we restrict our attention to the nontrivial case $1\le t\le n-1$.

The study of tilings (and packings) of $\mathbb{Z}^n$ by $\mathcal{B}(n,1,k_{+},k_{-})$ has attracted considerable attention in recent years, both for its intrinsic combinatorial interest and for its applications to flash memory. In particular, extensive research has been devoted to the cases of the cross $\mathcal{B}(n,1,k,k)$ and the semi-cross $\mathcal{B}(n,1,k,0)$ (see, for example, \cite{GS81,HS1974,HS1984,HS1986,KBE11,KLNY11,KLY12,S67,S1984,SS1994,S86,S87,T98,W1995}).

Subsequently, Schwartz \cite{S12} generalized these shapes to the quasi-cross $\mathcal{B}(n,1,k_{+},k_{-})$, which has since been the subject of extensive investigation (see \cite{S14,XL20,XL21,YKB13,YZZG20,ZG16,ZG18,ZZG17} and the references therein).

In contrast, for the case $t\ge2$, only a limited number of results are known (see \cite{BE13,GWX25,KLNY11,LTWZ25,S90,ZLG2023} and the references therein). This highlights a significant gap in our understanding of higher-weight error balls.

It is also worth noting that the cross $\mathcal{B}(n,1,1,1)$ coincides with the Lee sphere of radius $1$. More generally, the existence of tilings of $\mathbb{Z}^n$ by Lee spheres is a long-standing open problem, known as the Golomb–Welch conjecture \cite{GW1970}. Except for the cases of radius $1$ or dimension $2$, it is widely conjectured that such tilings do not exist. We refer the reader to \cite{E2011,H2009,H2009-2,HA2012,LZ2020,P1975,ZG2017,ZG2024,ZZ2019}, as well as the  works \cite{E22,HK2018}, for surveys on the current status of this conjecture.

In this paper, we focus on lattice tilings of $\mathbb{Z}^n$ by semi-cross $\mathcal{B}(n,1,k,0)$. It is well known that this geometric problem can be reduced to an algebraic one, namely the problem of splitting finite abelian groups.

To this end, let $G$ be a finite abelian group, written additively. For $s\in G$ and a nonnegative integer $m\in\mathbb{Z}$, we denote by $ms$ the sum $s+\cdots+s$, where $s$ appears $m$ times.

\begin{definition}
    A \emph{splitting} of $G$ consists of a pair of sets $M\subset\mathbb{Z}\setminus\{0\}$, called the \emph{multiplier set}, and $S=\{s_1,s_2,\dots,s_n\}\subseteq G$, called the \emph{splitter set}, such that every nonzero element of $G$ admits a unique representation of the form $ms$, where $m\in M$ and $s\in S$.
\end{definition}

Let $k$ be a positive integer, and define
\[S(k):=\{1,2,\dots,k\}.\]
The connection between tilings and group splittings is given by the following classical result.

\begin{theorem}\cite{HS1986,SS1994}
There exists a lattice tiling of $\mathbb{Z}^n$ by the semi-cross $\mathcal{B}(n,1,k,0)$ if and only if there exists a finite abelian group $G$ of order $nk+1$ such that $S(k)$ splits $G$.
\end{theorem}

Moreover, the above problem can be further reduced to the cyclic case.

\begin{theorem}\cite{H1983,SS1994}
If $S(k)$ splits a finite abelian group $G$, then it also splits the cyclic group $\mathbb{Z}_{|G|}$.
\end{theorem}

Consequently, the existence of lattice tilings by semi-crosses is closely tied to the structure of splittings of cyclic groups. In particular, the following necessary condition is known.

\begin{theorem}\cite[Chapter 4, Theorem 3]{SS1994}\label{thm2}
Let $n\ge3$ and $k$ be integers. If $S(k)$ splits a finite abelian group $G$ of order $nk+1$, then $k\le n-2$.
\end{theorem}

To further investigate splittings of finite abelian groups, we recall the following definition from \cite{SS1994}.

\begin{definition}
Let $G$ be a finite abelian group, and let $M$ and $S$ be the multiplier and splitter sets forming a splitting of $G$. The splitting is called \emph{nonsingular} if $\gcd(|G|,m)=1$ for every $m\in M$. Otherwise, it is called \emph{singular}. Moreover, if for every prime divisor $p$ of $|G|$ there exists some $m\in M$ such that $p\mid m$, then the splitting is called \emph{purely singular}.
\end{definition}

This terminology provides a useful framework for the study of splittings, since the classification problem for splittings of finite abelian groups can be reduced to the classification of nonsingular and purely singular splittings \cite{SS1994}. Although a number of results are known for nonsingular splittings, the purely singular case remains much less understood. One result established by Woldar \cite{W1995} is the following.

\begin{theorem}\label{lem-1}\cite{W1995}
If $S(k)$ splits a finite abelian group $G$ purely singularly, then either $\gcd(|G|,6)=1$ or $G$ is one of $\mathbb{Z}_1$, $\mathbb{Z}_{k+1}$, and $\mathbb{Z}_{2k+1}$.
\end{theorem}

He further proposed the following conjecture, which, according to Woldar, was verified by Hickerson for all $k\le 3000$ \cite{W1995}.

\begin{conjecture}\cite{W1995}\label{conj1}
If $S(k)$ splits a finite abelian group $G$ purely singularly, then $G$ is one of $\mathbb{Z}_1$, $\mathbb{Z}_{k+1}$, or $\mathbb{Z}_{2k+1}$.
\end{conjecture}

The main goal of this paper is to prove Conjecture~\ref{conj1}. In view of Theorem~\ref{thm2}, it is sufficient to prove the following theorem.

\begin{theorem}\label{mainthm}
Let $n$ and $k$ be integers. If $S(k)$ splits $\mathbb{Z}_{nk+1}$ purely singularly, then $k\ge n$.
\end{theorem}

\section{Proof of Theorem~\ref{mainthm}}
In this section, we prove our main result, namely Theorem~\ref{mainthm}. We begin by recalling the following lemma, which will play an important role in the proof.

\begin{lemma}\cite{S87}\label{lemma-1}
If $p$ is an odd prime and $\mathbb{Z}_{p^t}\setminus\{0\}=M\cdot S$ is a splitting, then either $M$ or $S$ consists entirely of elements relatively prime to $p$.
\end{lemma}

Suppose that \(S(k)\) splits the cyclic group \(G\) in a purely singular manner. By Theorem~\ref{lem-1}, we may further assume that $\gcd(|G|,6)=1$. Let \(p\) be the smallest prime divisor of \(|G|\), and write \(|G| = p^\alpha m\), where \(\gcd(p,m)=1\) and \(m \ge 1\). Let \(S=\{s_1,\dots,s_n\}\) be a splitter set corresponding to this splitting. Then
\[
p^\alpha m = kn + 1.
\]

For each integer \(i\), we define
\[
G_i = \{ x \in G : p^{i} \parallel \circ(x) \} \quad \text{and} \quad S_i = S \cap G_i.
\]
Here \(\circ(x)\) denotes the order of \(x\) in \(G\), and \(p^{i} \parallel \circ(x)\) means that \(p^i \mid \circ(x)\) but \(p^{i+1} \nmid \circ(x)\).

We first consider the elements of \(G_\alpha\). Counting these elements via the splitting yields
\[
\left(k - \left\lfloor \frac{k}{p} \right\rfloor\right) |S_\alpha|
= p^{\alpha-1}(p-1)m.
\]
It follows that there exist integers $0\leq \beta\leq \alpha -1, d,$ and $m'$ such that 
\[
k - \left\lfloor \frac{k}{p} \right\rfloor = p^\beta d m',
\]
where $d = \gcd\!\left(k - \left\lfloor \frac{k}{p} \right\rfloor,\, p-1\right)$ and \(m' \mid m\).
Consequently, there exist integers \(0 \le l \le l'\), \(1\leq \beta_i \le \alpha_i\), and distinct primes $p_i$ for \(1 \le i \le l'\) such that
\[
m = \prod_{i=1}^{l'} p_i^{\alpha_i}
\quad \text{and} \quad
m' = \prod_{i=1}^{l} p_i^{\beta_i}.
\]

As $(p^{\beta}m')\mid(p^{\alpha}m)$, there exists a unique subgroup \(M'\) of \(G\) of order \(p^{\beta}m'\). Without loss of generality, assume that
\[
S \cap \mathbb{Z}_{p^\alpha m}^{*} = \{s_1, s_2, \dots, s_r\}.
\]
For each \(1 \le i \le r\), define
\[
W_i = \{\, s_i + x : x \in M',\ s_i + x \in \mathbb{Z}_{p^\alpha m}^{*} \,\}.
\]
Then
\[
|W_i| = p^{\beta} \prod_{j=1}^{l} u_j,
\]
where
\[
u_j =
\begin{cases}
p_j^{\beta_j}, & \text{if } \beta_j < \alpha_j, \\[4pt]
p_j^{\beta_j - 1}(p_j - 1), & \text{if } \beta_j = \alpha_j.
\end{cases}
\]
Let \(T = [1,d]\), and define
\[T W_i = \{\, t \omega : 1 \le t \le d,\ \omega\in W_i \,\}.\]
Since \(d<p<p_i\) for all \(1\le i\le l'\), we have
$\{1, 2,\ldots, d\}\subset \mathbb{Z}_{p^{\alpha}m}^{*}$. Consequently,
$|TW_i|=d|W_i|$. Therefore,
\begin{equation}\label{ea}
|TW_i|\geq \left\{\begin{array}{ll} dp^{\beta}+ dp^{\beta}\prod_{j=1}^{\ell} p_j^{\beta_j-1}(p_j-1), &  \mbox{ if } \beta_j<\alpha_j\mbox{ for some } 1\leq j\leq \ell,\\
 dp^{\beta}\prod_{j=1}^{\ell} p_j^{\beta_j-1}(p_j-1), & \mbox{ otherwise.}
 \end{array} \right.\end{equation}

\begin{lemma}\label{lemma--1}
	\begin{itemize}
		\item[(a)] \(TW_i \cap TW_j = \emptyset\) for all distinct \(1 \le i,j \le r\).
		\item[(b)] $\bigcup_{i=1}^{r} TW_i \subseteq \mathbb{Z}_{p^\alpha m}^{*},$
		and  $r|TW_i|\le|\mathbb{Z}_{p^\alpha m}^{*}|.$
	\end{itemize}
\end{lemma}

\begin{proof}
	(a) Suppose, to the contrary, that there exist \(t_1,t_2\in T\) and \(x_1,x_2\in M'\) such that
	\[
	t_1(x_1+s_i)=t_2(x_2+s_j).
	\]
	Multiplying both sides by \(p^\beta m'\), we obtain
	\[
	p^\beta m' t_1(x_1+s_i)
	=
	p^\beta m' t_2(x_2+s_j).
	\]
	Since \(p^\beta m' x_1=p^\beta m' x_2=0\) in \(G\), it follows that
	\[
	p^\beta m' t_1 s_i
	=
	p^\beta m' t_2 s_j.
	\]
	Observe that $1\le p^\beta m' t_1,\ p^\beta m' t_2 \le k.$
	Hence the above equality contradicts the uniqueness of representations in the splitting. Therefore, $TW_i \cap TW_j=\emptyset$
	whenever \(i\ne j\).
	
	(b) By the definition of \(W_i\), we have $W_i\subseteq \mathbb{Z}_{p^\alpha m}^{*}.$
	Since both \(T\) and \(W_i\) are subsets of \(\mathbb{Z}_{p^\alpha m}^{*}\), it follows that $\bigcup_{i=1}^{r} TW_i
	\subseteq\mathbb{Z}_{p^\alpha m}^{*}.$
	The inequality $r|TW_i|\le|\mathbb{Z}_{p^\alpha m}^{*}|$ now follows immediately from part~(a).
\end{proof}

Next, define the following subsets of integers:
\begin{align*}
A &= \{x \in [1,k] : \gcd(x, \textstyle\prod_{i=1}^{l} p_i) = 1\},\\
B &= \{x \in [1,\lfloor k/p \rfloor] : \gcd(x, \textstyle\prod_{i=1}^{l} p_i) = 1\},\\
C &= \{x \in [1, k - \lfloor k/p \rfloor] : \gcd(x, \textstyle\prod_{i=1}^{l} p_i) = 1\},\\
D &= \{x \in [1,k] : \gcd(x, p \textstyle\prod_{i=1}^{l} p_i) = 1\},\\
E &= \{x \in [1,k] : \gcd(x, |G|) = 1\}.
\end{align*}

It is immediate to see that $A = (d p^{\beta} m' + B) \cup C$ and $(d p^{\beta} m' + B) \cap C = \emptyset,$
and hence \(|A| = |B| + |C|\).  
Recalling that \(k - \left\lfloor \frac{k}{p} \right\rfloor = p^{\beta} d m'\), we then obtain $|C| = p^{\beta} d \prod_{i=1}^{l} p_i^{\beta_i - 1}(p_i - 1).$ Observe that
\begin{align*}
A \setminus D
&= \{x \in [1,k] : p \mid x,\ \gcd(x/p, \textstyle\prod_{i=1}^{l} p_i) = 1\} \\
&= \{p y : y \in [1,\lfloor k/p \rfloor],\ \gcd(y, \textstyle\prod_{i=1}^{l} p_i) = 1\} \\
&= \{p y : y \in B\}.
\end{align*}
Therefore, $|A \setminus D|=|B|$ and consequently, \[
|D| = |A| - |A \setminus D| = |A| - |B| = |C|
= p^{\beta} d \prod_{i=1}^{l} p_i^{\beta_i - 1}(p_i - 1).
\]
In view of (\ref{ea}) and  the definitions of $D$ and $E$, we deduce that 
\[
|T W_i| =p^{\beta}d \prod_{i=1}^{\ell} p_i^{\beta_i-1}(p_i-1) \ge |D| \ge |E|.
\]
Moreover, note that \( |D|>|E| \) if \(l<l'\), and that $|TW_i|>|D|$  if  $ \beta_j<\alpha_j$ for  some $ 1\leq j\leq \ell$. 
On the other hand, by Lemma~\ref{lemma--1},
\[
|\mathbb{Z}_{p^{\alpha} m}^{*}|
\ge \left| \bigcup_{i=1}^{r} T W_i \right|
\ge r |D|
\ge r |E|=|\mathbb{Z}_{p^{\alpha} m}^{*}|.
\]
It follows that equality must hold throughout. In particular, $|T W_i| = |D| = |E|$. 
By the observations above, we conclude that \(p,p_1,\dots,p_l\) are precisely the prime divisors of \(|G|\), that is, $\ell=\ell'$, and moreover,  
\(\alpha_i = \beta_i\) for all \(i=1,\ldots, \ell\). 
Consequently, \[
|G| = p^{\alpha} \prod_{i=1}^{l} p_i^{\alpha_i},\  k - \left\lfloor \frac{k}{p} \right\rfloor = p^{\beta} d m \mbox{ and } 
|S_\alpha| = p^{\alpha - 1 - \beta} \left(\frac{p - 1}{d}\right).
\]

To complete the proof of Theorem~\ref{mainthm}, we first show that $\beta=\alpha-1.$

\begin{lemma}
\[
k - \left\lfloor \frac{k}{p} \right\rfloor = p^{\alpha-1} d m.
\]
\end{lemma}

\begin{proof}
Recall that
\[
k - \left\lfloor \frac{k}{p} \right\rfloor = p^{\beta} d m
\]
for some \(\beta \le \alpha - 1\). It remains to show that \(\beta = \alpha - 1\). The case \(\alpha = 1\) is trivial, so we may assume \(\alpha \ge 2\).
Write the base-\(p\) expansion of \(k\) as
\[
k = b_r p^r + \cdots + b_1 p + b_0,
\]
where \(0 \le b_i \le p-1\) and \(b_r \neq 0\). Since \(p^{\beta} \parallel \bigl(k - \left\lfloor \frac{k}{p} \right\rfloor\bigr)\), we have
\[
b_{\beta} = b_{\beta-1} = \cdots = b_1 = b_0,
\]
and if \(r > \beta\), then \(b_{\beta+1} \ne b_{\beta}\).

We claim that
\[
|S_{\alpha-1}| = \cdots = |S_{\alpha-\beta}| = 0.
\]
By counting elements in \(G_{\alpha}\) and \(G_{\alpha-1}\) via the splitting, we obtain
\begin{equation}\label{e1}
\left(k - \left\lfloor \frac{k}{p} \right\rfloor\right) |S_{\alpha}|
= p^{\alpha-1}(p-1)m,
\end{equation}
\begin{equation}\label{e2}
\left(k - \left\lfloor \frac{k}{p} \right\rfloor\right) |S_{\alpha-1}|
+ \left(\left\lfloor \frac{k}{p} \right\rfloor - \left\lfloor \frac{k}{p^2} \right\rfloor\right) |S_{\alpha}|
= p^{\alpha-2}(p-1)m.
\end{equation}
Multiplying \eqref{e2} by \(p\) and comparing with \eqref{e1}, we obtain
\[
p\left(k - \left\lfloor \frac{k}{p} \right\rfloor\right) |S_{\alpha-1}|
= \left(k - \left\lfloor \frac{k}{p} \right\rfloor - p \left\lfloor \frac{k}{p} \right\rfloor + p \left\lfloor \frac{k}{p^2} \right\rfloor\right) |S_{\alpha}|
= (b_0 - b_1) |S_{\alpha}| = 0,
\]
since \(b_0 = b_1\). Hence \(|S_{\alpha-1}| = 0\).

Proceeding inductively, assume that \(|S_{\alpha-1}| = \cdots = |S_{\alpha-u}| = 0\) for some \(u < \beta\). Since $1\leq \alpha-\beta\leq \alpha -u-1$, by considering elements in  \(G_{\alpha-u-1}\), we obtain
\begin{equation}\label{e4}
\left(k - \left\lfloor \frac{k}{p} \right\rfloor\right) |S_{\alpha-u-1}|
+ \left(\left\lfloor \frac{k}{p^{u+1}} \right\rfloor - \left\lfloor \frac{k}{p^{u+2}} \right\rfloor\right) |S_{\alpha}|
= p^{\alpha-u-2}(p-1)m.
\end{equation}
Note that by assumption $|S_{\alpha-1}|= \cdots = |S_{\alpha-u}| = 0$. 
Arguing as before, we obtain
\begin{align*}
p^{u+1}\left(k - \left\lfloor \frac{k}{p} \right\rfloor\right) |S_{\alpha-u-1}|
&= \left(k - \left\lfloor \frac{k}{p} \right\rfloor - p^{u+1} \left\lfloor \frac{k}{p^{u+1}} \right\rfloor + p^{u+1} \left\lfloor \frac{k}{p^{u+2}} \right\rfloor\right) |S_{\alpha}| \\
&= ((b_up^u+\cdots+b_0)-(b_{u+1}p^u+\cdots+b_1))|S_{\alpha}| = 0,
\end{align*}
since \(b_0=b_1=\cdots = b_{u+1}\). Thus \(|S_{\alpha-u-1}| = 0\), proving the claim.

\medskip

\noindent\textbf{Case (1): \(m > 1\).}
In this case, $k>p^{\beta+1}$, and hence \(b_{\beta+1} \ne b_0\). If \(\alpha > \beta+1\), then by considering the elements in \(G_{\alpha-\beta-1}\) and arguing as before, we obtain
\begin{align*}
p^{\beta+1}\left(k - \left\lfloor \frac{k}{p} \right\rfloor\right) |S_{\alpha-\beta-1}|
&= \left(k - \left\lfloor \frac{k}{p} \right\rfloor - p^{\beta+1} \left\lfloor \frac{k}{p^{\beta+1}} \right\rfloor + p^{\beta+1} \left\lfloor \frac{k}{p^{\beta+2}} \right\rfloor\right) |S_{\alpha}| \\
&= p^{\beta} (b_0 - b_{\beta+1}) |S_{\alpha}|.
\end{align*}
Substituting $k - \left\lfloor \frac{k}{p} \right\rfloor=p^{\beta}dm$ into the above identity yields
\[
p^{\beta+1}dm|S_{\alpha-\beta-1}| = (b_0 - b_{\beta+1}) |S_{\alpha}|.
\]
Since \(p < p_1\) and \(1 \le |b_0 - b_{\beta+1}| \le p-1\), it follows that \(p_1 \mid |S_{\alpha}|\). On the other hand,
\[
|S_{\alpha}| d = (p-1) p^{\alpha-\beta-1},
\]
which implies that \(p_1 \nmid |S_{\alpha}|\), a contradiction. Therefore, \(\beta = \alpha - 1\).

\medskip

\noindent\textbf{Case (2): \(m = 1\).}
In this case, \(|G| = p^{\alpha}\), where \(p\) is an odd prime. By Lemma~\ref{lemma-1} and the fact that $p\in \{1,2,\ldots, k\}$, the splitter set \(S\) must consist entirely of elements of order \(p^\alpha\). Hence,  \(|S_{\alpha}| = |S|\). Since $k|S|=kn=p^{\alpha}-1$, we have $p\nmid |S_{\alpha}|$. However, 
$|S_\alpha| = p^{\alpha - 1 - \beta} \left(\frac{p - 1}{d}\right)$, which is impossible unless $\beta=\alpha-1$. This completes the proof.
\end{proof}

\bigskip

\noindent\textit{Proof of Theorem~\ref{mainthm}.}
Recall that \(p^\alpha m = kn + 1\) and
\[
k - \left\lfloor \frac{k}{p} \right\rfloor = p^{\alpha-1} d m.
\]
Since $|G|\neq p$,  either $\alpha\ge 2$ or $m>1$.  In either case, 
\[ k\ge p^{\alpha-1} d m>p>\frac{|G|-1}{k}=n.\]
This completes the proof of Theorem~\ref{mainthm}.

\section*{Acknowledgements}
Ka Hin Leung was partially supported by National Natural Science Foundation of China (Grant No. 12131011).
Tao Zhang is partially supported by National Natural Science Foundation of China (Grant No. 12571357),
Natural Science Basic Research Program of Shaanxi  (Program No. 2025JC-YBMS-048).

\end{document}